\documentclass{article}

\usepackage{amsmath,amsthm,amssymb,enumerate}

\newtheorem{thm}{Theorem}[section]

\newtheorem{proposition}[thm]{Proposition}

\newtheorem{claim}[thm]{Claim}

\theoremstyle{definition}
\newtheorem{definition}[thm]{Definition}

\newcommand{\pr}{\mathbb{P}}

\newcommand{\Z}{\mathbb{Z}}
\newcommand{\Q}{\mathbb{Q}}
\newcommand{\R}{\mathbb{R}}
\newcommand{\C}{\mathbb{C}}

\newcommand{\NE}{\operatorname{NE}}

\newcommand{\Rat}{\operatorname{RatCurves}^n}
\newcommand{\Bl}{\operatorname{Bl}}

\newcommand{\Locus}{\operatorname{Locus}}

\newcommand{\sO}{\mathcal{O}}
\newcommand{\sN}{\mathcal{N}}

\title{Fano manifolds which are not slope stable along curves}

\author{Kento Fujita}

\begin{document}
\maketitle
\begin{abstract}{\noindent We show that a Fano manifold $(X,-K_X)$ is \emph{not} slope stable 
with respect to a smooth curve $Z$ if and only if $(X,Z)$ is isomorphic to one of 
(projective space, line), (product of projective line and projective space, fiber of second projection) or 
(blow up of projective space along linear subspace of codimension two, nontrivial fiber of blow up). 
}
\end{abstract}

\section{Introduction}
In this paper, we refine the result of Hwang, Kim, Lee and Park \cite{ylee} on the 
slope stability of a Fano manifold with respects to smooth curves. 

Let $X$ be a Fano manifold, that is, a smooth projective variety 
whose anticanonical divisor $-K_X$ is ample. It is conjectured that the $K$-polystability of $(X,-K_X)$ is 
equivalent to the existence of a K\"ahler-Einstein metric, 
but it is hard to investigate the $K$-(poly, semi)stability in general. 
On the other hand, Ross introduced the notion of the slope stability for polarized varieties (see \cite{RT}). 
The slope stability is weaker than the $K$-stability, and it is easier to investigate. 

Recently, Hwang, Kim, Lee and Park \cite{ylee} studied the slope stability of Fano manifolds and proved the following. 
If a Fano $n$-fold $(X,-K_X)$ is not slope stable with respect to a smooth curve $Z$, 
then $Z$ is a rational curve whose Seshadri constant $\epsilon(Z,X)$ is larger than $n-1$. 
Moreover, the normal bundle $\sN_{Z|X}$ is 
either trivial or $\sO_{\pr^1}^{\oplus n-2}\oplus\sO_{\pr^1}(-1)$ 
unless $Z$ is a line on a projective space (see Proposition \ref{yleesesh} and Theorem \ref{yleemain}). 

In this article, we give the complete classification of Fano manifolds which are not slope stable 
with respect to smooth curves using the results of \cite{ylee} and some classification results of Tsukioka \cite{tsu}. 

\begin{thm}[Main Theorem]\label{mainthm}
Let $X$ be a smooth Fano $n$-fold with $n\geq 3$ and let $Z\subset X$ be a smooth curve. 
Then we have the following.
\begin{enumerate}
\item
$(X,-K_X)$ is not slope stable with respect to $Z$ if and only if $(X,Z)$ is isomorphic to one of 
$(\pr^n,\text{line}),(\pr^1\times\pr^{n-1},\pr^1\times\text{pt})$ or $(\Bl_{\pr^{n-2}}\pr^n,\text{excp.line}).$
\item
$(X,-K_X)$ is not slope semistable with respect to $Z$ if and only if $(X,Z)$ is isomorphic to 
$(\Bl_{\pr^{n-2}}\pr^n,\text{excp.line}).$
\end{enumerate}
Here $\Bl_{\pr^{n-2}}\pr^n\rightarrow\pr^n$ is the blow up of $\pr^n$ along a linear subspace of codimension $2$ and 
\emph{excp.line} in $\Bl_{\pr^{n-2}}\pr^n$ is a nontrivial fiber of 
the blow up $\Bl_{\pr^{n-2}}\pr^n\rightarrow\pr^n$. 
\end{thm}

Note that the normal bundle $\sN_{Z|X}$ is trivial in the case $(\pr^1\times\pr^{n-1},\pr^1\times\text{pt})$ 
and $\sO_{\pr^1}^{\oplus n-2}\oplus\sO_{\pr^1}(-1)$ in the case $(\Bl_{\pr^{n-2}}\pr^n,\text{excp.line})$, respectively. 

To prove Theorem \ref{mainthm}, we study Fano $n$-folds $X$ with smooth curves $Z\subset X$ 
whose Seshadri constants $\epsilon(Z,X)$ are larger than $n-1$ in Proposition \ref{seshprop}. 
Then Theorem \ref{mainthm} follows from the proposition immediately. 
We also classify such pairs $(X,Z)$ in Theorem \ref{appthm} using the result of Proposition \ref{seshprop} 
and the classification result of Fano $3$-fold \cite{MoMu}. 

\smallskip

\noindent\textbf{Acknowledgements.}
The author is grateful to Professor Shigeru Mukai for making various suggestions which helped him to improve 
Proposition \ref{seshprop}. 
He also expresses his gratitude to Professor Shigefumi Mori for his warm encouragement. 
The author thanks both professors and Professor Kenji Matsuki for reading the preliminary version of the paper in detail. 
He also thanks Doctor Yuji Odaka for teaching him about the slope stability 
and Doctor Takuzo Okada for a careful reading the draft. 
The author is partially supported by JSPS Fellowships for Young Scientists.

\smallskip

\noindent\textbf{Notation and terminology.}
We always work over the complex number field $\C$. 
For the theory of extremal contraction, we refer the readers to \cite{KoMo}. 
For a smooth projective variety $X$ and a $K_X$-negative extremal ray $R\subset\overline{\NE}(X)$,
we define the \emph{length} $l(R)$ of $R$ by
$$\min\{(-K_X\cdot C)\mid C\text{ is a rational curve with } [C]\in R\}.$$
A rational curve $C\subset X$ with 
$[C]\in R$ and $(-K_X\cdot C)=l(R)$ is called \emph{a minimal rational curve of $R$}.

For a smooth projective variety $X$, we denote the normalization of 
the space of irreducible and reduced rational curves on $X$ by $\Rat(X)$ (see \cite[Definition II.2.11]{kollar}). 
\emph{A family $H$ of rational curves on $X$} always means an irreducible component of $\Rat(X)$. 
We define $\Locus(H)$ to be the locus $\bigcup_{[C]\in H}C$ of curves on $X$ parametrized by $H$. 
For a point $x\in X$, we define $\Locus(H_x)$ to be the locus $\bigcup_{[C]\in H, x\in C}C$ 
of curves on $X$ parametrized by $H$ and passing through $x$.

\medskip

\section{Preliminaries}

\smallskip

In this section, we review the definition of and results on the slope stability of Fano manifolds (for details, see \cite{ylee}).

Let $X$ be a Fano $n$-fold and $Z\subset X$ be a smooth closed subvariety. 
Let $\pi:\hat{X}\rightarrow X$ be the blow up along $Z\subset X$ and $E\subset\hat{X}$ be its exceptional divisor. 
For $k$ with $xk\in\Z_{>0}$ and $k\gg0$, we can write
$$\chi(\hat{X},{\pi}^*(-kK_X)-xkE)=\sum_{j=0}^{n}a_{n-j}(x)k^j,$$
where $a_i(x)\in\Q[x]$. Set $\tilde{a}_i(x):=a_i(0)-a_i(x)$. 

\begin{definition}\label{slopedef}
For $c\in(0,\epsilon(Z,X)]$, set 
$$\mu_c(\sO_Z,-K_X):=\frac{\int_0^c(\tilde{a}_1(x)+\frac{\tilde{a}'_0(x)}{2})dx}{\int_0^c\tilde{a}_0(x)dx}.$$
A Fano manifold $(X,-K_X)$ is \emph{slope stabe} (resp. \emph{slope semistable}) with respect to $Z$ if 
\begin{eqnarray*}
\mu_c(\sO_Z,-K_X)>n/2\ (\text{resp. }\geq n/2)\ 
\text{ for all }c\in(0,\epsilon(Z,X)],
\end{eqnarray*}
where $\epsilon(Z,X)$ is \emph{the Seshadri constant}
$$\max\{c\in{\R}_{\geq 0}|\pi^*(-K_X)-cE \text{ numerically effective}\}.$$
\end{definition}

\begin{proposition}[{\cite[Lemma 2.10]{ylee}}]\label{yleesesh}
Let $X$ be a Fano manifold and $Z$ be a smooth closed subvariety of codimension $r$. 
If $\epsilon(Z,X)\leq r$, then $(X,-K_X)$ is slope stable with respect to $Z$.
\end{proposition}

\begin{thm}[{\cite[Theorem 1.2(1)]{ylee}}]\label{yleemain}
Let $X$ be a Fano $n$-fold with $n\geq 3$. 
If $(X,-K_X)$ is not slope stable with respect to a smooth curve $Z$, 
then the curve $Z$ is one of the following 
\begin{enumerate}[$(a)$]
\item
a rational curve whose normal bundle is $\sO_{\pr^1}^{\oplus n-1}$ and $\epsilon(Z,X)=n$;
\item
a rational curve whose normal bundle is $\sO_{\pr^1}^{\oplus n-2}\oplus\sO_{\pr^1}(-1)$;
\item
a line on $\pr^n$.
\end{enumerate}
\end{thm}

\medskip

\section{Proof of the Main Theorem}

\smallskip

This section is devoted to the proof of Theorem \ref{mainthm}. 

\begin{proposition}\label{seshprop}
Let $X$ be a Fano $n$-fold with $n\geq 3$, let $Z\subset X$ be a smooth curve 
and let $\pi:\hat{X}\rightarrow X$ be the blow up along $Z\subset X$. Assume that $\epsilon(Z,X)>n-1$ holds. 
Then either holds:
\begin{enumerate}[$(a)$]
\item
All $K_{\hat{X}}$-negative extremal rays $R$ with $(E\cdot R)>0$ have of fiber type, or 
\item
there exists a prime divisor $D\subset X$ with $D\simeq\pr^{n-1}$ and $\sN_{D|X}\simeq\sO_{\pr^{n-1}}$ such that 
$D$ contains $Z$ as a line. 
\end{enumerate}
\end{proposition}

We will improve this result in Theorem \ref{appthm}. 

\begin{proof}
Let $E\subset\hat{X}$ be the exceptional divisor of $\pi$. 
$E$ and $Z$ are Fano manifolds by \cite[Lemma 2.11]{ylee}. Hence we have $Z\simeq\pr^1$ and $E\simeq\pr^1\times\pr^{n-2}$ 
or $\pr_{\pr^1}(\sO_{\pr^1}^{\oplus n-2}\oplus\sO_{\pr^1}(1))$ by \cite[Lemma 2.13]{ylee}. 
$\hat{X}$ is a Fano manifold since $-K_{\hat{X}}=\pi^*(-K_X)-(n-2)E$. 
Furthermore, we have $\epsilon(E,\hat{X})=\epsilon(Z,X)-(n-2)>1$. 
Since $\hat{X}$ is a Fano $n$-fold, 
there exists an extremal ray $R\subset\NE(\hat{X})$ with a minimal rational curve $[C]\in R$ such that $(E\cdot C)>0$. 
We have $(-K_{\hat{X}}\cdot C)\geq 2$ since $(-K_{\hat{X}}\cdot C)\geq\epsilon(E,\hat{X})(E\cdot C)>(E\cdot C)>0$. 
Let $\phi:\hat{X}\rightarrow Y$ be the contraction of $R$.

Assume that $\phi$ is birational. 
It is enough to show that the condition of $(b)$ holds under the assumption. 
Pick a family of rational curves $H\subset\Rat(\hat{X})$ with $[C]\in H$. $H$ is projective 
since $C$ is minimal. Take an arbitrary point $x\in E\cap\Locus(H)$. 
Assume that there exists an irreducible component $P\subset\phi^{-1}\phi(x)$ such that $\dim P\geq 3$. 
Then $E$ must intersect $P$ since $(E\cdot C)>0$, hence there exists a point $z\in Z$ such that $P\cap\pi^{-1}(z)\neq\emptyset$. 
Thus we have $\dim(P\cap\pi^{-1}(z))\geq 1$ since $\dim\pi^{-1}(z)=n-2$, which is a contradiction since 
no curve on $\hat{X}$ is contracted by both $\pi$ and $\phi$. Hence we have $\dim\phi^{-1}\phi(x)\leq 2$. 
We also have $\dim\Locus(H_x)\geq n-\dim\Locus(H)+(-K_{\hat{X}} \cdot C)-1$ by \cite[Proposition 2.5(a)]{ACO}. 
Hence we have 
\begin{eqnarray*}
2 &\geq& \dim\phi^{-1}\phi(x)\geq\dim\Locus(H_x)\\
&\geq& n-\dim\Locus(H)+(-K_{\hat{X}} \cdot C)-1\\
&\geq&(-K_{\hat{X}} \cdot C) \geq 2
\end{eqnarray*}
since $\dim\Locus(H)\leq n-1$. 
Thus we have 
$\dim\Locus(H)=n-1$ and $(-K_{\hat{X}}\cdot C)=2$. 
Therefore $\phi$ is a divisorial contraction and $\Locus(H)$ is the exceptional divisor of $\phi$, 
which we denote by $F$ from now on. 
Furthermore, we have $(E \cdot C)=1$ and $l(R)=2$ holds. 
Thus $\dim\phi^{-1}\phi(x)=2$ holds for all $x\in E\cap F$. 
It follows that $Y$ is a smooth projective variety and $\phi$ is the blow up 
along a smooth closed subvariety $B\subset Y$ with codimension $3$ by \cite[Theorem 5.1]{AO}. 
We have $E\neq F$ since $(E\cdot C)>0$. 
Set $E':=\phi(E)$. 
Then $\phi|_E:E\rightarrow E'$ is not an isomorphism but a birational morphism and 
$E'$ is a normal projective variety with $B\subset E'$ since $\dim(E\cap\phi^{-1}(y))\geq 1$ for all $y\in B$ 
and $E'$ is a divisor on $Y$ with regular in codimension $1$. 
Hence we have $E\simeq\pr_{\pr^1}(\sO_{\pr^1}^{\oplus n-2}\oplus\sO_{\pr^1}(1))$, 
$E'\simeq\pr^{n-1}$ and $\phi|_E:E\rightarrow E'$ 
is the blow up along a linear subspace $B\subset E'$ since $E$ is either isomorphic to 
$\pr^1\times\pr^{n-2}$ or $\pr_{\pr^1}(\sO_{\pr^1}^{\oplus n-2}\oplus\sO_{\pr^1}(1))$. 
Take a curve $l\subset E$ which is a line in a fiber of $\pi\simeq\pr^{n-2}$. 
Then we have $(E'\cdot\pi_*l)=0$ since $(E\cdot l)=-1$ and $(F\cdot l)=1$. 
Hence $\sN_{E'|Y}\simeq\sO_{\pr^{n-1}}$. 
We have $F\simeq\pr_{\pr^{n-3}}(\sO_{\pr^{n-3}}^{\oplus 2}\oplus\sO_{\pr^{n-3}}(1))$ 
since $\sN_{B|E'}\simeq\sO_{\pr^{n-3}}(1)^{\oplus 2}$. Set $D:=\pi(F)$. 
Repeating the same argument of the case $\phi|_{E}:E\rightarrow E'$, we have $D\simeq\pr^{n-1}$ and 
$\pi|_{F}:F\rightarrow D$ is a blow up along a line $Z\subset D\simeq\pr^{n-1}$. 
Take a curve $m\subset F$ which is a line in a fiber of $\phi\simeq\pr^2$. 
Then we have $(D\cdot\phi_*m)=0$ since $(F\cdot m)=-1$ and $(E\cdot m)=1$. 
Thus we have $\sN_{D|X}\simeq\sO_{\pr^{n-1}}$. 
Therefore we have proved the proposition. 
\end{proof}

Now, we prove the ``if" part of Theorem \ref{mainthm}. 
Assume $(X,-K_X)$ is not slope stable with respect to $Z$. 
By Proposition \ref{yleesesh}, $\epsilon(Z,X)>n-1$ holds. 
If $\phi:\hat{X}\rightarrow Y$ constructed as above is birational then 
$\sN_{Z|X}\simeq\sO_{\pr^1}(1)^{\oplus n-2}\oplus\sO_{\pr^1}$ by $(b)$ of Proposition \ref{seshprop}, 
which is a contradiction to Theorem \ref{yleemain}. Thus $\phi$ is of fiber type. 

If $n\geq 4$, this situation of contraction morphisms $\pi$ and $\phi$ is exactly the case studied by Tsukioka. 
We remark that $\dim Y\geq n-2$ since $2=n-\dim\pi^{-1}\pi(x)\geq\dim\phi^{-1}\phi(x)$ holds for all $x\in E$. 
If $\dim Y=n-2$, then $(X,Z)$ is isomorphic to either $(\pr^n,\text{line})$ 
or $(\Q^n,\text{conic})$ by \cite[Proposition 3]{tsu}. 
If $\dim Y=n-1$, then 
$(X,Z)$ is isomorphic to one of $(\Q^n,\text{line})$, $(\pr^1\times\pr^{n-1},\pr^1\times\text{pt})$, 
$(\Bl_{\pr^{n-2}}\pr^n,$ line (disjoint from $\pr^{n-2}$)$)$ or 
$(\Bl_{\pr^{n-2}}\pr^n,\text{excp.line})$  by \cite[Proposition 4]{tsu}. 
Assume $(X,Z)$ is isomorphic to one of $(\Q^n,\text{conic})$, $(\Q^n,\text{line})$ 
or $(\Bl_{\pr^{n-2}}\pr^n,$ line (disjoint from $\pr^{n-2}$)$)$. 
Then $\sN_{Z|X}$ is isomorphic to 
$\sO_{\pr^1}(2)^{\oplus n-1}$, $\sO_{\pr^1}(1)^{\oplus n-2}\oplus\sO_{\pr^1}$ 
or $\sO_{\pr^1}(1)^{\oplus n-1}$, respectively. 
Therefore $(X,-K_X)$ is slope stable with respect to $Z$ by Theorem \ref{yleemain}, whch is a contradiction. 

Now, we consider the case $n=3$. 
Assume $\rho_X=1$. If $l(R)=2$ then we have $\epsilon(Z,X)=r-2\leq 2$, where $r$ is a index of $X$ by \cite[Theorem 5.1]{MoMu83}. 
Hence we have $l(R)=3$ and $(X,Z)$ is isomorphic to $(\pr^3,\text{line})$. 
Now, we assume $\rho_X\geq 2$. 
We have $\sN_{Z|X}\simeq\sO_{\pr^1}^{\oplus 2}$ or $\sO_{\pr^1}\oplus\sO_{\pr^1}(-1)$ by Theorem \ref{yleemain}. 
We can see that $\phi$ is a $\pr^1$-bundle and $E$ is a section of $\phi$ since $l(R)>(E\cdot C)>0$ and $\rho_Y\geq 2$. 
If $\sN_{Z|X}\simeq\sO_{\pr^1}^{\oplus 2}$, then we have $\sN_{E|\hat{X}}\simeq\sO_{\pr^1\times\pr^1}(-1,0)$. Hence we have 
an exact sequence 
$$0\rightarrow\sO_{\hat{X}}\rightarrow\sO_{\hat{X}}(E)\rightarrow\sN_{E|\hat{X}}\rightarrow 0.$$
Thus we obtain 
$$0\rightarrow\sO_{\pr^1\times\pr^1}\rightarrow\phi_*\sO_{\hat{X}}(E)\rightarrow\sO_{\pr^1\times\pr^1}(-1,0)\rightarrow 0.$$
Therefore $(X,Z)$ is isomorphic to $(\pr^1\times\pr^2,\pr^1\times\text{pt})$ since 
$\hat{X}\simeq\pr_{\pr^1\times\pr^1}(\phi_*\sO_{\hat{X}}(E))\simeq
\pr_{\pr^1\times\pr^1}(\sO_{\pr^1\times\pr^1}\oplus\sO_{\pr^1\times\pr^1}(-1,0))$. 
If $\sN_{Z|X}\simeq\sO_{\pr^1}\oplus\sO_{\pr^1}(-1)$, then $(X,Z)$ is isomorphic to 
$(\Bl_{\text{line}}\pr^3,\text{excp.line})$ by the same technique as we have seen above. 
As a consequence, we have completed ``if part" of the proof of Theorem \ref{mainthm}. 

Now, we prove the converse. 
If $(X,Z)$ is isomorphic to $(\pr^n,\text{line})$, then $(X,-K_X)$ is not slope stable 
but slope semistable with respect to $Z$ by \cite[Remark 3.5]{ylee}. 
If $(X,Z)$ is isomorphic to $(\pr^1\times\pr^{n-1},\pr^1\times\text{pt})$, 
then $(X,-K_X)$ is not slope stable but slope semistable with respect to $Z$ by \cite[Example 3.8]{ylee}. 

Now, we consider the case where $(X, Z)$ is isomorphic to $(\Bl_{\pr^{n-2}}\pr^n,\text{excp.line})$. 
By \cite[Proposition 3.1(ii)]{ylee}, 
it is enough to show $\epsilon(Z,X)\geq n$ and $(-K_X\cdot Z)=1$ to see that 
$(X,-K_X)$ is not slope semistable with respect to $Z$. 
We can see $(-K_X\cdot Z)=1$ immediately. 
Let $\sigma$ be the blow up $\sigma:=\Bl_{\pr^{n-2}}:X\rightarrow\pr^n$ and $F$ be its exceptional divisor, 
and let $\pi$ be the blow up $\pi:=\Bl_Z:\hat{X}\rightarrow X$ and $E$ be its exceptional divisor. 
Then the complete linear system 
$|\pi^*(-K_X)-nE|=|\pi^*(\sigma^*\sO_{\pr^n}(1)-F)+n(\pi^*\sigma^*\sO_{\pr^n}(1)-E)|$ 
on $\hat{X}$ is base point free since so are both complete linear systems 
$|\pi^*(\sigma^*\sO_{\pr^n}(1)-F)|$ and $|\pi^*\sigma^*\sO_{\pr^n}(1)-E|$. 
In particular, $\pi^*(-K_X)-nE$ is numerically effective. 
Therefore, $(X,-K_X)$ is not slope semistable with respect to $Z$. 

Hence we have completed the proof of Theorem \ref{mainthm}.

\medskip

\section{Curves on Fano manifolds with large Seshadri constants}

\smallskip

In this section, we make Proposition \ref{seshprop} into the final form. 
We use the classification result of Fano $3$-fold \cite{MoMu} to prove Theorem \ref{appthm} in dimension $3$. 

\begin{thm}\label{appthm}
Let $X$ be a Fano $n$-fold with $n\geq 3$ and let $Z\subset X$ be a smooth curve. 
Assume that $\epsilon(Z,X)>n-1$ holds. Then $(X,Z)$ is isomorphic to one of 
$(\pr^n,\text{line})$,$(\Q^n,\text{conic})$,$(\Q^n,\text{line})$,
$(\pr^1\times\pr^{n-1},\pr^1\times\text{pt})$,
$(\Bl_{\pr^{n-2}}\pr^n,$ line $($disjoint from $\pr^{n-2}))$
 or $(\Bl_{\pr^{n-2}}\pr^n,\text{excp.line}).$
\end{thm}

\begin{proof}
Let $\pi:\hat{X}\rightarrow X$ and $\phi:\hat{X}\rightarrow Y$ be same as in Proposition \ref{seshprop}. 
If $\phi$ is of fiber type, then we have seen  in the proof of Proposition \ref{seshprop} 
that we get the result of Theorem \ref{appthm} for the case $n\geq 4$. 
We also get the same result of Theorem \ref{appthm} for the case $n=3$ by the classification result \cite{MoMu}, 
but we omit the proof. 
Hence we can assume that $\phi$ is birational. We have seen in Proposition \ref{seshprop} 
that the condition of Proposition \ref{seshprop} $(b)$ holds. 

First, we consider the case $n\geq 4$. 
There exists an extremal ray $R_X\subset\NE(X)$ such that $(D \cdot R_X)>0$. 
We denote the contraction of $R_X$ by $\sigma:X\rightarrow W$. 
If there exists $w\in W$ such that $\dim\sigma^{-1}(w)\geq 2$ holds, then $D\cap\sigma^{-1}(w)$ contains a curve. 
However, all curves in $D\simeq\pr^{n-1}$ are numerically proportional, hence $D$ must be contracted to a point by $\sigma$. 
Hence $W$ must be a point since $(D\cdot R_X)>0$. Thus we have $\rho_X=1$ but 
this leads to a contradiction since $\sN_{D|X}\simeq\sO_{\pr^{n-1}}$. Therefore, $\dim\sigma^{-1}(w)\leq 1$ for all $w\in W$. 
Hence $\sigma$ is either of fiber type or a divisorial contraction by \cite[Corollary p.\ 145]{wisn}. 

\begin{claim}
There exists an irreducible curve $g\subset X$ such that $g$ is contracted by $\sigma$, $g\neq Z$ 
and $g\cap Z\neq\emptyset$. 
\end{claim}

\begin{proof}
If $\sigma(Z)=\text{pt}$ then $\sigma(D)=\text{pt}$ since $D\simeq\pr^{n-1}$, 
which is a contradiction as we see above. 
Hence it is enough to show the existence of a nontrivial fiber $g\subset X$  of $\sigma$ with 
$g\cap Z\neq\emptyset$. 
This is obvious in the case $\sigma$ is of fiber type. 
We assume that $\sigma$ is a divisorial contraction. 
We denote the exceptional divisor by $G$. 
We can see that $G\cap D\subset D\simeq\pr^{n-1}$ contains a divisor in $D\simeq\pr^{n-1}$ since $G$ intersects $D$. 
Therefore $G\cap D$ must intersects $Z$. 
\end{proof}

For such a $g\subset X$, let $\hat{g}\subset\hat{X}$ be the strict transform of $g\subset X$. However we have 
$$0<(-K_{\hat{X}}\cdot\hat{g})=(-K_X\cdot g)-(n-2)(E\cdot\hat{g})\leq 0$$
since $n-2\geq 2$ and $(-K_X\cdot g)\leq 2$ by \cite[Theorem (1.1)]{wisn}, 
which is a contradiction. Thus we have proved Theorem \ref{appthm} for the case $n\geq 4$. 

Now, we consider the case $n=3$. 
$\phi:\hat{X}\rightarrow Y$ is a blow up at a smooth point and the complete linear system $|E'|$ on $Y$ 
gives a surjective morphism $Y\rightarrow\pr^1$ since $\sN_{E'|Y}\simeq\sO_{\pr^2}$. 
Then we have $Y\simeq\pr_{\pr^1}(\sO_{\pr^1}^{\oplus 2}\oplus\sO_{\pr^1}(1))$ by the classification result \cite[p.\ 160]{MoMu}. 
Hence $(X,Z)$ is isomorphic to $(\pr^1\times\pr^2,$pt$\times$line$)$. However, we have $\epsilon(Z,X)=2$ in this situation, 
which is a contradiction. 

Therefore we have completed the proof of Theorem \ref{appthm}. 
\end{proof}

\smallskip

\noindent K.\ Fujita\\
Research Institute for Mathematical Sciences (RIMS),\\
Kyoto University, Oiwake-cho, Kitashirakawa, Sakyo-ku, Kyoto 606-8502, Japan \\
fujita@kurims.kyoto-u.ac.jp

\end{document}